\documentclass[12pt,reqno]{amsart}
\usepackage{amsmath, amsfonts, amssymb, amsthm, amscd, amsbsy}
\usepackage{fancyhdr}
\usepackage[usenames,dvipsnames,svgnames,x11names,hyperref]{xcolor}
\usepackage{geometry}
\usepackage{graphicx}
\usepackage[pagebackref]{hyperref}%
\usepackage{amsmath}%
\usepackage{amsfonts}%
\usepackage{amssymb}
\usepackage{enumerate}

\providecommand{\U}[1]{\protect\rule{.1in}{.1in}}
\hypersetup{backref=true,
pagebackref=true,
hyperindex=true,
colorlinks=true,
breaklinks=true,
urlcolor=NavyBlue,
linkcolor=Fuchsia,
bookmarks=true,
bookmarksopen=false,
filecolor=black,
citecolor=ForestGreen,
linkbordercolor=red
}

\allowdisplaybreaks

\newtheorem{theorem}{Theorem}[section]

\newtheorem{lemma}[theorem]{Lemma}
\newtheorem{remark}[theorem]{Remark}
\newtheorem{example}[theorem]{Example}
\newtheorem{examples}[theorem]{Examples}
\newtheorem{foo}[theorem]{Remarks}

\def\vint{\mathop{\mathchoice%
          {\setbox0\hbox{$\displaystyle\intop$}\kern 0.22\wd0%
           \vcenter{\hrule width 0.6\wd0}\kern -0.82\wd0}%
          {\setbox0\hbox{$\textstyle\intop$}\kern 0.2\wd0%
           \vcenter{\hrule width 0.6\wd0}\kern -0.8\wd0}%
          {\setbox0\hbox{$\scriptstyle\intop$}\kern 0.2\wd0%
           \vcenter{\hrule width 0.6\wd0}\kern -0.8\wd0}%
          {\setbox0\hbox{$\scriptscriptstyle\intop$}\kern 0.2\wd0%
           \vcenter{\hrule width 0.6\wd0}\kern -0.8\wd0}}%
          \mathopen{}\int}

\begin{document}
\title[High-order Lane-Emden equations on the sphere $\mathbb{S}^n$]{Sharp subcritical Sobolev inequalities and uniqueness of nonnegative solutions to high-order Lane-Emden equations   on $\mathbb{S}^n$    }

\author{Lu Chen}
\address{School of Mathematics and Statistics, Beijing Institute of Technology, Beijing 100081, P. R. China}
\email{chenlu5818804@163.com}

\author{Guozhen Lu}
\address{Department of Mathematics\\
University of Connecticut\\
Storrs, CT 06269, USA}
\email{guozhen.lu@uconn.edu}

\author{Yansheng Shen}
\address{School of Mathematical Sciences, Beijing Normal University, Beijing 100875, P.R. China.}
\email{ysshen@mail.bnu.edu.cn}

\thanks{The first author was partly supported by NSFC (No.11901031). The second author was partly supported by a grant from the Simons Foundation.}

\begin{abstract}
In this paper, we are concerned with the uniqueness result for non-negative solutions of the higher-order Lane-Emden equations involving the GJMS operators on $\mathbb{S}^n$. Since the classical moving-plane method based on the Kelvin transform and maximum principle fails in dealing with the high-order elliptic equations in $\mathbb{S}^n$, we first employ the Mobius transform between $\mathbb{S}^n$ and $\mathbb{R}^n$, poly-harmonic average and iteration arguments to show that the higher-order Lane-Emden equation on $\mathbb{S}^n$ is equivalent to some integral equation in $\mathbb{R}^n$. Then we apply the method of moving plane in integral forms and the symmetry of sphere to obtain the uniqueness of nonnegative solutions to the higher-order Lane-Emden equations with subcritical polynomial growth on $\mathbb{S}^n$. As an application, we also identify the best constants and classify the extremals of the sharp subcritical high-order Sobolev inequalities involving the GJMS operators on $\mathbb{S}^n$. Our results do not seem to be in the literature even for the Lane-Emden equation and sharp subcritical Sobolev inequalities for first order derivatives on $\mathbb{S}^n$.
\end{abstract}

\maketitle

\maketitle {\small {\bf Keywords:} Uniqueness of nonnegative solutions; Lane-Emden equations; Sharp constants; Extremal functions; Moving-plane method; GJMS operators on $\mathbb{S}^n.$ \\

{\bf 2010 MSC.} 5J30; 46E35; 35B06; 35A02. }

\section{Introduction and main results}
The main purpose of this paper is to discuss the uniqueness of non-negative solutions of the high-order Lane-Emden equations in $\mathbb{S}^n$ and establish the sharp subcritical high-order Sobolev inequality involving the GJMS operator in $\mathbb{S}^n$. For $m\in \mathbb{N}^{*}$, $1\leq p<+\infty$, PDEs of the form
\begin{equation}\label{sem}
(-\Delta)^m u=u^{p}\ \ {\rm in}\ \mathbb{R}^n\, \ \ \
\end{equation}
are called higher-order Lane-Emden equations. These equations have numerous important applications in conformal geometry,  mathematical physics
and astrophysics et al., we refer the readers to e.g. \cite{CH,CGS, Gazzola, Stru} and the references therein.
\medskip

When $m=1$, Gidas, Ni and Nirenberg in \cite {GNN} proved that any positive solution of the Lane-Emden equation \eqref{sem} in the critical case $p=\frac{n+2}{n-2}$ satisfying $u(x)=O(\frac{1}{|x|^{n-2}})$ at infinity must take the form as
$$u(x)=\frac{\big(n(n-2)\lambda^2\big)^{\frac{n-2}{4}}}{\big(\lambda^2+|x-x_0|^2\big)^{\frac{n-2}{2}}},$$
for some $\lambda>0$ and $x_0\in \mathbb{R}^n$. Then,
Caffarelli, Gidas and Spruck \cite{CGS} removed the decay condition at infinity and obtained the same result by employing the Kelvin-transform. It should be mentioned that the classification of the solutions of Lane-Emden equations in the critical case plays an important role in the study of the well-known Yamable problem and the prescribing scalar curvature problem. In the subcritical case $1\leq p<\frac{n+2}{n-2}$, Gidas and Spruck \cite{GS} proved that any nonnegative solution of the Lane-Emden equations must be equal to zero. Later, Chen and Li in \cite {CL1} gave a  simpler proof of this Liouville type result.
\medskip

When $m=2$, $n>4$ and $1\leq p<\frac{n+4}{n-4}$, Lin \cite{Lin} applied the method of moving plane to prove that all nonnegative $C^4(\mathbb{R}^n)$ smooth solutions must be equal to zero. Wei and Xu \cite{WX} generalized Lin's
results to the cases when $m$ is any integer.

\medskip

It is well known that $-\Delta$, $(-\Delta)^2$, $\cdot\cdot\cdot (-\Delta)^m$ are the high-order conformal Laplace operators in $\mathbb{R}^n$.
Let $(M,g)$ be a complete Riemannian manifold of dimension $n$. Then the conformal
Laplacian or the Yamabe operator $P_{1,g}$ on Riemannian manifold $(M,g)$ is defined by $P_{1,g}=-\Delta_{g}+\frac{n-2}{4(n-1)}R_g$, where $R_{g}$ is the scalar curvature of  $(M,g)$. A fourth order conformally invariant operator $P_{2,g}$ with the leading term
$(-\Delta)^2_{g}$ was discovered by Paneitz \cite{Pan} and later Branson \cite{Bra} found a conformal sixth order operator $P_{3,g}$ with the
leading term $(-\Delta)^3_{g}$. The existence of a general conformal operator $P_{m,g}$ of order $2m$ was obtained by Graham, Jenne, Mason and Sparling \cite{GJMS} and such an operator is  known as the GJMS operator. (see also \cite{J}) From their work, we know that if $(M,g)$ is a Riemannian manifold of even dimension $n$, then for $k=1,2,\cdot\cdot\cdot,\frac{n}{2}$, there exists a conformally invariant differential operator $P_{k,g}$ of the form
$P_{k,g}=(-\Delta)^{k}_g+$ lower order terms, satisfying for a conformal metric $\tilde{g}=e^{2u}g$,
$$P_{k,\tilde{g}}(v)=e^{-(\frac{n}{2}+k)u}P_{k,g}(e^{(\frac{n}{2}-k)u}v).$$
It is an interesting fact pointed out in \cite{G} and \cite{GH} that when $n$ is even and $k>\frac{n}{2}$, a conformal operator $P_{k,g}$ may not exist. Hence, $P_{\frac{n}{2},g}$ is usually called as the critical GJMS operator.
\medskip

In the spirit of the aforementioned works in the Euclidean spaces, a natural question arises: Does there exist a Liouville-type result for the following Lane-Emden equation on a complete Riemannian manifold $(M,g)$ with dimension $n>2m$
\begin{equation}\label{hlane}
P_{m,g}(u)=u^{p},\ u\geq 0 \ in \ M,\ \
\end{equation}
in the subcritical case $1\leq p<\frac{n+2m}{n-2m}?$ It is well-known that the method based on the Kelvin-transform and moving plane fails in dealing with the elliptic equations on the general Riemannian manifold. In this paper, we make an attempt to answer this question in the special manifold: the unit sphere $\mathbb{S}^n$. In $S^n$ $(n>2m)$, it follows from Branson's work in \cite{Bra1} that the explicit expression of the conformal operator $P_{m,g_{\mathbb{S}^n}}$ of order $2m$ is given by
$$P_{m,g_{\mathbb{S}^n}}=P_{1,g_{\mathbb{S}^n}}(P_{1,g_{\mathbb{S}^n}}-2)(P_{1,g_{\mathbb{S}^n}}-6)\cdot\cdot\cdot(P_{1,g_{\mathbb{S}^n}}-m(m-1)),$$
where $P_{1,g_{\mathbb{S}^n}}=-\Delta_{g_{\mathbb{S}^n}}+\frac{n(n-2)}{4}$ and $\Delta_{g_{\mathbb{S}^n}}$ denotes the Laplace-Beltrami operator in $\mathbb{S}^n$. (See also Hang \cite{Hang1})
\vskip0.1cm

The first result of the present paper is the following:
\begin{theorem}\label{thm1}
Assume that $u$ is a nonnegative $C^{2m}(\mathbb{S}^n)$ solution satisfying the elliptic equation
\begin{equation}\label{hlane}
P_{m,g_{\mathbb{S}^n}}(u)=u^{p}\  {\rm in}\ \mathbb{S}^n, \ \ 1\leq p<\frac{n+2m}{n-2m},
\end{equation}
then $u$ must be equal to some constant.
\end{theorem}

\begin{remark}
Our method is based on the Mobius transform between $\mathbb{S}^n$ and $\mathbb{R}^n$, poly-harmonic average, iteration arguments and symmetry of the sphere. We first employ the Mobius transform to convert the high-order Lane-Emden equation \eqref{hlane} in $\mathbb{S}^n$ into the high-order elliptic equation in $\mathbb{R}^n$. Then by establishing the super poly-harmonic properties of nonnegative solutions, we derive the equivalence between the high-order elliptic equation in $\mathbb{R}^n$ and some integral equation in $\mathbb{R}^n$. Finally, by applying the method of moving-plane in integral forms developed by Chen, Li and Ou \cite{CLO} combining with the symmetry of sphere, we conclude that any non-negative solution of high-order Lane-Emden equation \eqref{hlane} in $\mathbb{S}^n$  must be equal to some constant.
\end{remark}

Slightly modifying the proof of Theorem \ref{thm1}, we can obtain the following general Liouville-type theorem.

\begin{theorem}\label{thmp1}
 Let $f(t)=\sum_{k=1}^{l}a_kt^{p_k}$ with $a_k\geq0$ and $1\leq p_1\leq p_2\leq \cdot\cdot\cdot\leq p_{k-1}<p_k=\frac{n+2m}{n-2m}$ and assume that $u$ is a nonnegative $C^{2m}(\mathbb{S}^n)$ solution satisfying the elliptic equation
\begin{equation}\label{hlane1}
P_{m,g_{\mathbb{S}^n}}(u)=f(u)\  {\rm in}\ \mathbb{S}^n.
\end{equation}
Then $u$ must be equal to some constant.
\end{theorem}

Aubin (\cite{Aub1}, \cite{Aub3}) and Talenti \cite{Tal} established the sharp Sobolev inequality for the first order derivatives:
$$\int_{\mathbb{R}^n}|\nabla u|^2dx\geq S_{n} \big(\int_{\mathbb{R}^n}|u|^{\frac{2n}{n-2}}dx\big)^{\frac{n-2}{n}},$$
where $S_n=\frac{n(n-2)}{4}|\mathbb{S}^n|^{\frac{2}{n}}$ denotes the sharp constant of the above Sobolev inequality. The sharp constant $S_n$ can be achieved if and only if $u$ takes the form as $$u(x)=\frac{1}{(1+|x-x_0|^2)^{\frac{n-2}{2}}}$$
up to some translation and dilation. In order to solve the famous Yamabe problem, Aubin (see \cite{Aub1, Aub2}) also established sharp first order Sobolev inequality in $W^{1,2}(\mathbb{S}^n)$:
$$\int_{\mathbb{S}^n}|\nabla_{g_{\mathbb{S}^n}}u|^2d\sigma+\frac{n(n-2)}{4}\int_{\mathbb{S}^n}|u|^2d\sigma\geq \frac{n(n-2)}{4}|\mathbb{S}^n|^{\frac{2}{n}}\big(\int_{\mathbb{S}^n}|u|^{\frac{2n}{n-2}}d\sigma\big)^{\frac{n-2}{n}},$$
where $d\sigma$ denotes the surface measure in $\mathbb{S}^{n}$.
Later, Beckner\cite{Be} obtained the sharp Sobolev inequality for high-order derivatives in $W^{m,2}(\mathbb{S}^n)$ $(n>2m)$:
$$\int_{\mathbb{S}^n}P_{m,g_{\mathbb{S}^n}}(u)ud\sigma\geq S_{n,m} \big(\int_{\mathbb{S}^n}|u|^{\frac{2n}{n-2m}}d\sigma\big)^{\frac{n-2m}{n}},$$
where $P_{m,g_{\mathbb{S}^n}}$ denotes $2m$th-order conformal Laplacian of $u$ in $\mathbb{S}^n$ and $$S_{n,m}=\frac{n(n-2)}{4}(\frac{n(n-2)}{4}-2)(\frac{n(n-2)}{4}-6)\cdot\cdot\cdot(\frac{n(n-2)}{4}-m(m-1))|\mathbb{S}^n|^{\frac{2m}{n}}.$$
By conformal transformation between $\mathbb{S}^n$ and $\mathbb{R}^n$, the above inequality is in fact equivalent to the high-order Sobolev inequality in $\mathbb{R}^n$: $$\int_{\mathbb{R}^n}|\nabla^m u|^2dx\geq S_{n,m} \big(\int_{\mathbb{R}^n}|u|^{\frac{2n}{n-2m}}dx\big)^{\frac{n-2m}{n}},$$
with equality holds if and only if $u$ takes the form as $$u(x)=\frac{1}{(1+|x-x_0|^2)^{\frac{n-2m}{2}}}$$
up to some translation and dilation. In $\mathbb{R}^n$, the existence of extremals of the high-order Sobolev inequality can be obtained by the technique of symmetry and rearrangement or concentration-compactness principle (see \cite{Lieb} and \cite{Lions}) and the sharp constant can also be calculated by classifying the non-negative solutions of the high-order  equations
$$(-\Delta)^m=u^{\frac{n+2m}{n-2m}}.$$
We refer the reader to the works of, e.g.,  X. J. Wang \cite{Wang}, Swanson \cite{Swanson}, R. C. A. M. Van der Vorst \cite{Van},  Djadli,  Hebey and Ledoux \cite{DHL},   A. Cotsiolis and N. Tavoularis \cite{Cos}.
Obviously, in the case $n>2m$, $\big(\int_{\mathbb{S}^n}P_{m,g_{\mathbb{S}^n}}(u)ud\sigma\big)^{\frac{1}{2}}$ is also equivalent to the $m$th-order Sobolev norm in $W^{m,2}(\mathbb{S}^n)$. By the Sobolev compact embedding theorem, the sharp constant of subcritical high-order Sobolev inequality: $$S_{m,n,p}=\inf_{u\in W^{m,2}(\mathbb{S}^n)}\frac{\int_{\mathbb{S}^n}P_{m,g_{\mathbb{S}^n}}(u)ud\sigma}{\big(\int_{\mathbb{S}^n}|u|^{p}d\sigma\big)^{\frac{2}{p}}}$$
for $1<p<\frac{n+2m}{n-2m}$ can be achieved. However, so far, nothing is known for the accurate value of the sharp constant $S_{m,n,p}$ and precise forms of extremal functions. In this paper, as an application of our Theorem \ref{thm1}, we can classify the extremals of the above subcritical Sobolev inequality and identify the corresponding sharp constant $S_{m,n,p}$.
\begin{theorem}\label{thm2}
The following sharp subcritical Sobolev inequality holds for all $u\in W^{m,2}(\mathbb{S}^n)$:
$$
\big(\int_{\mathbb{S}^n}|u|^{p}d\sigma\big)^{\frac{2}{p}}\le S_{m, n. p}
 \int_{\mathbb{S}^n}P_{m,g_{\mathbb{S}^n}}(u)ud\sigma,  \ \ 2< p<\frac{2n}{n-2m}.
$$
If $u$ is an extremal function of  the above inequality, then  $u$ must be equal to some constant.
Moreover, the best constant
\begin{equation}\label{sub}
S_{m,n,p}=\inf_{u\in W^{m,2}(\mathbb{S}^n)}\frac{\int_{\mathbb{S}^n}P_{m,g_{\mathbb{S}^n}}(u)ud\sigma}{\big(\int_{\mathbb{S}^n}|u|^{p}d\sigma\big)^{\frac{2}{p}}}
\end{equation}
 is equal to
$$S_{m,n,p}=\frac{\frac{n(n-2)}{4}(\frac{n(n-2)}{4}-2)(\frac{n(n-2)}{4}-6)\cdot\cdot\cdot(\frac{n(n-2)}{4}-m(m-1))|\mathbb{S}^n|}{|\mathbb{S}^n|^{\frac{2}{p}}}.$$
\end{theorem}
\begin{remark}
It should be noted that Theorem \ref{thm2} is not a direct consequence of Theorem \ref{thm1}. Although an extremal of the functional associated with
 the high-order Sobolev inequality \eqref{sub} also satisfies the equation \eqref{hlane}, it is not easy to obtain the nonnegativity of the extremal.
  It is well known that the non-negativity of the extremals for general geometric inequality in Riemannian manifold $(M,g)$ for first order derivative
   is straightforward. This is mainly because for any $u\in W^{1,p}(M)$, one can derive
   $|u|\in W^{1,p}(M)$ with $\|\nabla_g(|u|)\|_{L^p(M)}\leq \|\nabla_gu\|_{L^p(M)}$. However, this property cannot hold for high-order
    derivatives. Hence the extremal of a geometric inequality of high-order derivatives may not be non-negative. Nevertheless, in
     our setting we are able to prove that the extremals of the high-order Sobolev inequality \eqref{sub} associated with the GJMS operator
     are non-negative. Our method is based on the equivalence between the high-order Sobolev inequality \eqref{sub} and the Hardy-Littlewood-Sobolev
     inequality in $\mathbb{S}^n$ (see Lieb \cite{Lieb}).
\end{remark}

\begin{remark}
Assume that $u_{p}$ is an extremal of the sharp subcritical high-order Sobolev inequality \eqref{sub}, then it is easy to check that
$$\lim_{p\rightarrow \frac{2n}{n-2m}}\frac{\int_{\mathbb{S}^n}P_{m,g}(u_{p})u_{p}d\sigma}{\big(\int_{\mathbb{S}^n}|u_{p}|^{p}d\sigma\big)^{\frac{2}{p}}}=
 \inf_{u\in W^{m,2}(\mathbb{S}^n)}\frac{\int_{\mathbb{S}^n}P_{m,g}(u)ud\sigma}{\big(\int_{\mathbb{S}^n}|u|^{\frac{2n}{n-2m}}d\sigma\big)^{\frac{n-2m}{n}}}.$$
 Hence in some sense, the sharp constant of critical geometric inequality can be deduced from the sharp constant of subcritical geometric inequality by
 the limiting procedure. This method of calculating the sharp constant of the critical geometric inequality is called subcritical approximation approach (One can  refer to \cite{Hang} for such an approach).
\end{remark}

\begin{remark}
When $m=1$, the above sharp subcritical high-order Sobolev inequality reduces to the following subcritical Sobolev inequality:
$$\int_{\mathbb{S}^n}|\nabla_{g_{\mathbb{S}^n}}u|^2d\sigma+\frac{n(n-2)}{4}\int_{\mathbb{S}^n}|u|^2d\sigma\geq \frac{n(n-2)}{4}|\mathbb{S}^n|^{1-\frac{2}{p}}\big(\int_{\mathbb{S}^n}|u|^{p}d\sigma\big)^{\frac{2}{p}}.$$
In \cite{Be}, Beckner applied the technique of spherical harmonic decomposition to obtain the sharp subcritical Sobolev inequality for first order derivative in $\mathbb{S}^n$. Our method is based on the moving-plane method in integral forms, different from Beckner's method and can be used to establish the sharp subcritical Sobolev inequality for high-order derivatives.
\end{remark}

The organization of the paper is as follows. In Section \ref{section2}, we will prove that the non-negative solutions of high-order Lane-Emden equation \eqref{hlane} must be equal to some constant, namely we shall give the proof of Theorem \ref{thm1}.
Section \ref{section3} establishes best constants and the classification of the extremals for the subcritical high-order Sobolev inequality \eqref{sub}
in $\mathbb{S}^n$, namely the proof of Theorem \ref{thm2}.

\section{Proof of Theorem \ref{thm1}}\label{section2}

In this section, we shall prove that the non-negative solutions of high-order Lane-Emden equation \eqref{hlane} must be equal to some constant, namely we shall give the proof of Theorem \ref{thm1}. We first use the Mobius transform between $\mathbb{S}^n$ and $\mathbb{R}^n$ to transform the high-order Lane-Emden equation \eqref{hlane} on $\mathbb{S}^n$ into the high-order elliptic equation in $\mathbb{R}^n$.
\medskip

Consider the Mobius transform $\phi: x\in\mathbb{R}^n\rightarrow \xi \in \mathbb{S}^n\setminus \{(0,0,\cdot\cdot\cdot,-1)\}$ given by
$$\xi_{i}=\frac{2x_i}{1+|x|^2}\ \ {\rm for}\ i=1,2,\cdot\cdot\cdot,n;\ \ \xi_{n+1}=\frac{1-|x|^2}{1+|x|^2}.$$ It is well known that
$\phi$ is a conformal transform satisfying  $\phi^{*}(g_{\mathbb{S}^n})=\big(\frac{2}{1+|x|^2}\big)^2dx^2$ (see Lieb \cite{Lieb}), where $dx^2$ denotes standard Euclidean measure in $\mathbb{R}^n$. Hence it follows from the basic property of conformal transform that for any $v\in C^{2m}(\mathbb{S}^n)$,
$$P_{m,g_{\mathbb{S}^n}}(v)=\Big(\frac{2}{1+|x|^2}\Big)^{-\frac{n}{2}-m}(-\Delta)^m\big((\frac{2}{1+|x|^2})^{\frac{n}{2}-m}v(\phi)\big).$$
Let $u(x)=(\frac{2}{1+|x|^2}\big)^{\frac{n}{2}-m}v(\phi)$, then the high-order Lane-Emden equation \eqref{hlane} can be written
as the following high-order elliptic equation in $\mathbb{R}^n$ for some positive dimensional constant $C_{n}$:
\begin{equation}\label{eq1}
\begin{cases}
(-\Delta)^mu=u^p\big(\frac{2}{1+|x|^2}\big)^{\frac{n+2m}{2}-p\frac{n-2m}{2}}\ \ {\rm in \  \mathbb{R}^n},\\
\ \ \ \ \ u(x)\le C_{n}(1+|x|^2)^{m-\frac{n}{2}}, \ \ \ \ \ \  { x\in \mathbb{R}^n}.
\end{cases}\end{equation}
\vskip0.1cm

Next, we want to prove that the nonnegative solutions of the high-order elliptic equation \eqref{eq1} in $\mathbb{R}^n$ must be radially decreasing about the origin for $1\leq p<\frac{n+2m}{n-2m}$. Once we obtain this result, then the nonnegative solutions of Lane-Emden equation \eqref{hlane} with the subcritical polynomial growth in $S^n$ must be symmetric about the south pole $(0,0,\cdot\cdot\cdot,-1)$ in $\mathbb{S}^n$. Using the symmetry of the sphere $\mathbb{S}^n$, we can similarly prove that the solution $u$ must be symmetric about any point in $\mathbb{S}^n$. Hence we conclude that the nonnegative solutions of Lane-Emden equation \eqref{hlane} with the subcritical polynomial growth in $\mathbb{S}^n$ must be equal to some constant, which accomplishes the proof of Theorem \ref{thm1}. Based on the above analysis, in order to obtain the desired result, we only need to show the symmetry of the nonnegative solutions of the high-order elliptic equation \eqref{eq1} for $1\leq p<\frac{n+2m}{n-2m}$. Our method is based on the moving plane in integral forms developed by Chen, Li and Ou \cite{CLO}. For this purpose, our first objective is to prove that any nonnegative solution of the high-order elliptic equation \eqref{eq1} in $\mathbb{R}^n$ must satisfy some integral equation and the following lemma will be crucially used in our proof.

\begin{lemma}\label{lem1}
Assume that $u\geq 0$ satisfies the high-order elliptic equation \eqref{eq1} and set $-\Delta u=u_1$, $(-\Delta)^2u=u_2$,$\cdot\cdot\cdot$, $(-\Delta)^{m-1}u=u_{m-1}$. If we furthermore suppose that $u_i\geq 0$ for $i=1,2, \cdot\cdot\cdot, m-1$, then $(u,u_1,\cdot\cdot\cdot, u_{m-1})$ satisfies the following integral system
\begin{equation}\label{bi3*}\begin{cases}
u(x)=c_{n}\int_{\mathbb{R}^n}\frac{u_1(y)}{|x-y|^{n-2}}dy,\  x\in \mathbb{R}^n, \\
u_{i}(x)=c_{n}\int_{\mathbb{R}^n}\frac{u_{i+1}(y)}{|x-y|^{n-2}}dy,\  x\in \mathbb{R}^n,\ i=1,2,\cdot\cdot\cdot, m-2,\\
u_{m-1}(x)=c_{n}\int_{\mathbb{R}^n}\frac{u^p(y)}{|x-y|^{n-2}}\big(\frac{2}{1+|y|^2}\big)^{\frac{n+2m}{2}-p\frac{n-2m}{2}}dy,\  x\in \mathbb{R}^{n},
\end{cases}\end{equation}
where $c_{n}=\frac{1}{n(n-2)v_n}$ and $v_n$ denotes the volume of unit ball in $\mathbb{R}^n$.
\end{lemma}

\begin{proof}
We first prove that for any $x\in \mathbb{R}^n$, $$u(x)=c_{n}\int_{\mathbb{R}^n}\frac{1}{|x-y|^{n-2}}u_1(y)dy.$$
To do this, we start with showing that
$$c_{n}\int_{\mathbb{R}^n}\frac{1}{|x-y|^{n-2}}u_1(y)dy<+\infty.$$
Set $v_{R}(x)=\int_{B_{R}}G_{R}(x,y)u_{1}(y)dy$, where
$G_{R}(x,y)$ is the Green function of Laplace operator of the ball $B_{R}$ with the Dirichlet boundary.
Obviously, $v_{R}$ is continuous and satisfies
\begin{equation*}\begin{cases}
-\Delta v_R=u_1,\ \ x\in B_{R},\\
~v_{R}(x)=0,\ \ \ x\in \partial B_{R}.
\end{cases}\end{equation*}
Let $w_{R}=u-v_R$, then
\begin{equation*}\begin{cases}
-\Delta w_R=0,\ \ x\in B_{R},\\
~w_{R}(x)\geq0,\ \ x\in \partial B_{R}.
\end{cases}\end{equation*}
According to the maximum principle, we know that
$$u(x)\geq v_{R}(x)=\int_{B_{R}}G_{R}(x,y)u_1(y)dy$$ for
any $x\in B_{R}$ and $R\geq 0$.
Letting $R\rightarrow+\infty$, we arrive at
$$u(x)\geq c_{n}\int_{\mathbb{R}^n}\frac{u_1(y)}{|x-y|^{n-2}}dy.$$
Now we let $\tilde{v}(x)=c_{n}\int_{\mathbb{R}^n}\frac{u_1(y)}{|x-y|^{n-2}}dy$, then $\tilde{v}(x)$ satisfies $-\Delta v=u_1$ in $\mathbb{R}^n$. Set $w=u-\tilde{v}$, we derive
\begin{equation*}\begin{cases}
-\Delta w=0,\ x\in \mathbb{R}^n,\\
w=u-\tilde{v}\geq 0.
\end{cases}\end{equation*}
Then it follows from the classical Liouville-type theorem that $u(x)=\tilde{v}(x)+c_0$ for some $c_0\geq 0$. From the second inequality in \eqref{eq1}, we know that $\lim_{|x|\rightarrow+\infty}u(x)=0$. Combining the above estimate, we get $c_0=0$ and then
$$u=\tilde{v}=c_{n}\int_{\mathbb{R}^n}\frac{u_1(y)}{|x-y|^{n-2}}dy.$$
Next, we prove that
$$u_1(x)=c_{n}\int_{\mathbb{R}^n}\frac{u_{2}(y)}{|x-y|^{n-2}}dy,\  x\in \mathbb{R}^n.$$
Similar to the proof of integral expression of $u$, we can use the Green representation formula, maximum principle and Liouville-type theorem for non-negative harmonic function to obtain that $$u_1(x)=c_{n}\int_{\mathbb{R}^n}\frac{u_{2}(y)}{|x-y|^{n-2}}dy+c_0$$ for some $c_0\geq 0$. We claim that $c_0=0$. Suppose not, then $u_1(x)\geq c_0$ and $\tilde{v}(x)=\int_{\mathbb{R}^n}\frac{u_1(y)}{|x-y|^{n-2}}dy=+\infty$, which is in contradiction with the fact $u=\tilde{v}=c_{n}\int_{\mathbb{R}^n}\frac{u_1(y)}{|x-y|^{n-2}}dy$. The proof for the integral expression of $u_{i}~(2\leq i\leq m-1)$ is similar to the proof of integral expression of $u_{1}$. Hence Lemma \ref{lem1} is proved.
\end{proof}

If we can prove that $u_i\geq 0$ for $i=1,\cdot\cdot\cdot, m-1$, then according to Lemma \ref{lem1}, we can write $u(x)$ as follows
\begin{align*}
u(x)=c_n^{m}\int_{\mathbb{R}^{n}}\cdot\cdot\cdot \int_{\mathbb{R}^n}&\frac{1}{|x-y_1|^{n-2}}\frac{1}{|y_1-y_2|^{n-2}}\cdot\cdot\cdot \frac{u^p(y_{m})}{|y_{m-1}-y_m|^{n-2}} \\
&\cdot\big(\frac{2}{1+|y_m|^2}\big)^{\frac{n+2m}{2}-p\frac{n-2m}{2}}dy_1\cdot\cdot\cdot dy_{m}.
\end{align*}
Exchanging the order of integration and using the fact (see Stein \cite{Stein}) $$\int_{\mathbb{R}^n}\int_{\mathbb{R}^n}\frac{1}{|y_i-y_{i+1}|^{n-2}}\frac{1}{|y_{i+1}-y_{i+2}|^{n-2}}dy_{i+1}dy_{i+2}=\frac{\pi^{\frac{n}{2}}\Gamma(\frac{n}{2}-2)}{\Gamma^2(\frac{n}{2}-1)}\int_{\mathbb{R}^n}\frac{1}{|y_i-y_{i+2}|^{n-4}}dy_{i+2},$$ we derive that up to some constant,
\begin{equation}\label{int1}
u(x)=\int_{\mathbb{R}^n}\frac{u^p(y)}{|x-y|^{n-2m}}\big(\frac{2}{1+|y|^2}\big)^{\frac{n+2m}{2}-p\frac{n-2m}{2}}dy.
\end{equation}

According to the above analysis, in order to obtain the desired integral expression \eqref{int1} for any nonnegative solution of the high-order elliptic equation \eqref{eq1}, we only need to prove that the any nonegative solution of equation \eqref{eq1} has the super poly-harmonic property. It is well-known that spherical mean is an useful tool in proving the super poly-harmonic properties of non-negative solutions to high-order elliptic equation. (see e.g. Lu, Wei and Xu \cite{LWX}.) We will apply this strategy together with the re-centers and iteration arguments to prove super poly-harmonic properties. Now we start the proof of super poly-harmonic property of equation \eqref{eq1}, namely we will prove that $u_i\geq 0$ for $i=1,2,\cdot\cdot\cdot, m-1$. In particular, we remark that from now on, for simplicity, we will omit the constant $2^{\frac{n+2m}{2}-p\frac{n-2m}{2}}$.
\medskip

\textbf{The proof of the super poly-harmonic properties:} We first prove that $u_{m-1}\geq0$. We argue this by contradiction.
Suppose that there exists $x^{1}\in\mathbb{R}^{n}$ such that
\begin{align}\label{cont-point}
u_{m-1}(x^{1})<0,
\end{align}
we will deduce a contradiction by the following four steps.
\vskip0.1cm

\textbf{Step 1}. Let
$$\bar{u}(r):=\frac{1}{|\partial B_{r}(x^{1})|}\int_{\partial B_{r}(x^{1})}u(x)d\sigma, \ \
\bar{u}_{i}(r):=\frac{1}{|\partial B_{r}(x^{1})|}\int_{\partial B_{r}(x^{1})}u_{i}(x)d\sigma,$$
be the spherical average of $u$ and $u_{i}~(i=1,2,\ldots,m-1)$ with respect to the center $x^{1}$, where $d\sigma$ denotes the surface measure on $\partial B_{r}(x^{1})$. Then for $r>0$, we have
\begin{equation}\label{Lap-sph}
\begin{cases}
-\Delta\bar{u}=\bar{u}_{1},  \\
-\Delta\bar{u}_{1}=\bar{u}_{2},  \\
\ \ \ \vdots \\
-\Delta\bar{u}_{m-2}=\bar{u}_{m-1}, \\
-\Delta\bar{u}_{m-1}=\overline{\Big\{\frac{u^{p}(x)}{(1+|x|^{2})^{\frac{n+2m}{2}-p\frac{n-2m}{2}}}\Big\}}.
\end{cases}
\end{equation}
From the last equality in \eqref{Lap-sph} coupled with the Jensen's inequality, we obtain that for any $r>0$,
\begin{align}\label{non-est}
-\Delta\bar{u}_{m-1}(r)=&\frac{1}{|\partial B_{r}(x^{1})|}\int_{\partial B_{r}(x^{1})}\frac{u^{p}(x)}
{(1+|x|^{2})^{\frac{n+2m}{2}-p\frac{n-2m}{2}}}d\sigma  \nonumber\\
\geq&\frac{1}{(1+(r+|x^{1}|)^{2})^{\frac{n+2m}{2}-p\frac{n-2m}{2}}}\cdot\frac{1}{|\partial B_{r}(x^{1})|}\int_{\partial B_{r}(x^{1})}u^{p}(x)d\sigma  \nonumber\\
\geq&\frac{1}{(1+(r+|x^{1}|)^{2})^{\frac{n+2m}{2}-p\frac{n-2m}{2}}}\Big(\frac{1}{|\partial B_{r}(x^{1})|}\int_{\partial B_{r}(x^{1})}u(x)d\sigma\Big)^{p}  \nonumber\\
=&\frac{\bar{u}^{p}(r)}{(1+(r+|x^{1}|)^{2})^{\frac{n+2m}{2}-p\frac{n-2m}{2}}}
\geq0,
\end{align}
which implies
\begin{align}\label{equ-der-est}
-\frac{1}{r^{n-1}}(r^{n-1}\bar{u}_{m-1}'(r))'=-\Delta\bar{u}_{m-1}(r)\geq0, \ \ \forall \ r>0,
\end{align}
consequently, $\bar{u}'_{m-1}(r)\leq0$ for all $r>0$. Then by \eqref{cont-point} and using the Lebesgue differential theorem, we derive that
\begin{align}\label{m-1-qua-est}
\bar{u}_{m-1}(r)\leq\bar{u}_{m-1}(0)=u_{m-1}(x^{1}):=-c_{0}<0, \ \ \forall \ r\geq0.
\end{align}
This together with the relation $-\Delta\bar{u}_{m-2}=\bar{u}_{m-1}$ gives that
\begin{align*}
-\frac{1}{r^{n-1}}(r^{n-1}\bar{u}_{m-2}'(r))'=\bar{u}_{m-1}(r)\leq-c_{0}, \ \ \forall \ r\geq0.
\end{align*}
Integrating both sides of the above inequality, we obtain $\bar{u}_{m-2}'(r)\geq\frac{c_{0}}{n}r$, then integrating again yields
\begin{align}\label{iter-qua-m-2-est}
\bar{u}_{m-2}(r)\geq\bar{u}_{m-2}(0)+\frac{c_{0}}{2n}r^{2}, \ \forall ~r\geq0.
\end{align}
Thus, we can find a large $r_{1}>0$ such that
\begin{align}\label{new-iter-poi}
\bar{u}_{m-2}(r_{1})\geq\frac{c_{0}}{4n}r_{1}^{2}:=c_{1}>0.
\end{align}

\textbf{Step 2}. Take a point $x^{2}$ with $|x^{2}-x^{1}|=r_{1}$ as the new center, and make spherical average of $\bar{u}$ and $\bar{u}_{i}$ at the new center $x^{2}$, i.e.,
$$\bar{\bar{u}}(r):=\frac{1}{|\partial B_{r}(x^{2})|}\int_{\partial B_{r}(x^{2})}\bar{u}(|x-x^{1}|)d\sigma,$$
and
$$\bar{\bar{u}}_{i}(r):=\frac{1}{|\partial B_{r}(x^{2})|}\int_{\partial B_{r}(x^{2})}\bar{u}_{i}(|x-x^{1}|)d\sigma,$$
where $i=1,2,\ldots,m-1$. Then similar to \eqref{Lap-sph}, we have that $(\bar{\bar{u}},\bar{\bar{u}}_{1},\cdots,\bar{\bar{u}}_{m-1})$ still satisfies
\begin{equation}\label{sec-sph-aver}
\begin{cases}
-\Delta \bar{\bar{u}}=\bar{\bar{u}}_{1},  \\
-\Delta\bar{\bar{u}}_{1}=\bar{\bar{u}}_{2},  \\
\ \ \ \vdots \\
-\Delta\bar{\bar{u}}_{m-2}=\bar{\bar{u}}_{m-1}, \\
-\Delta\bar{\bar{u}}_{m-1}=\overline{\overline{\Big\{\frac{u^{p}(x)}{(1+|x|^{2})^{\frac{n+2m}{2}-p\frac{n-2m}{2}}}\Big\}}}.
\end{cases}
\end{equation}

Clearly, from the last equality in \eqref{sec-sph-aver} together with the definition of the spherical average, we can easily obtain that $-\Delta\bar{\bar{u}}_{m-1}(r)\geq0$ for all $r>0$, which yields that
$$-\frac{1}{r^{n-1}}(r^{n-1}\bar{\bar{u}}'_{m-1}(r))'\geq0.$$
As a consequence, $\bar{\bar{u}}'_{m-1}(r)\leq0$ and by \eqref{m-1-qua-est}, we get
\begin{equation}\label{siter-est-m-1}
\bar{\bar{u}}_{m-1}(r)\leq\bar{\bar{u}}_{m-1}(0)=\bar{u}_{m-1}(|x^{2}-x^{1}|)=\bar{u}_{m-1}(r_{1})
\leq-c_{0}<0, \ \ \forall \ r\geq0.
\end{equation}
On the other hand, notice that $-\Delta\bar{\bar{u}}_{m-2}=\bar{\bar{u}}_{m-1}$, then by \eqref{siter-est-m-1} it follows that
$$-\frac{1}{r^{n-1}}(r^{n-1}\bar{\bar{u}}_{m-2}'(r))=\bar{\bar{u}}_{m-1}(r)\leq-c_{0},$$
which gives that
\begin{equation}\label{siter-est-m-2}
\bar{\bar{u}}_{m-2}(r)\geq\bar{\bar{u}}_{m-2}(0)+\frac{c_{0}}{2n}r^{2},
\end{equation}
for all $r\geq0$. In particular, from \eqref{new-iter-poi}, we know that
\begin{equation}\label{regu-sec}
\bar{\bar{u}}_{m-2}(0)=\bar{u}_{m-2}(|x^{2}-x^{1}|)=\bar{u}_{m-2}(r_{1})\geq\frac{c_{0}}{4n}r_{1}^{2}
=\frac{c_{0}}{4n}|x^{2}-x^{1}|^{2}=c_{1}>0.
\end{equation}

Combining \eqref{siter-est-m-2}, \eqref{regu-sec} and the relation $-\Delta\bar{\bar{u}}_{m-3}=\bar{\bar{u}}_{m-2}$, we deduce that
$$-\frac{1}{r^{n-1}}(r^{n-1}\bar{\bar{u}}'_{m-3}(r))'=\bar{\bar{u}}_{m-2}(r)\geq\bar{\bar{u}}_{m-2}(0)\geq
c_{1}>0,$$
which implies that
\begin{align*}
\bar{\bar{u}}_{m-3}(r)\leq\bar{\bar{u}}_{m-3}(0)-\frac{c_{1}}{2n}r^{2}, \ \forall~r\geq0.
\end{align*}
Therefore, there exists a large $r_{2}>0$ such that
\begin{align}\label{thnew-iter-poi}
\bar{\bar{u}}_{m-3}(r_{2}):=-c_{2}<0.
\end{align}

\textbf{Step 3}. Following the same argument as what we did in \textbf{Step 2}, we can take a point $x^{3}$ with $|x^{3}-x^{2}|=r_{2}$ as the new center and define the spherical average of $\bar{\bar{u}}$ and $\bar{\bar{u}}_{i}~(i=1,2,\ldots,m-1)$ by re-centers:
$$\bar{\bar{\bar{u}}}(r):=\frac{1}{|\partial B_{r}(x^{3})|}\int_{\partial B_{r}(x^{3})}\bar{\bar{u}}(|x-x^{2}|)d\sigma,$$
$$\bar{\bar{\bar{u}}}_{i}(r):=\frac{1}{|\partial B_{r}(x^{3})|}\int_{\partial B_{r}(x^{3})}\bar{\bar{u}}_{i}(|x-x^{2}|)d\sigma.$$
Then it follows from \eqref{thnew-iter-poi} that
\begin{align}\label{thiter-ori-val}
\bar{\bar{\bar{u}}}_{m-3}(0)=\bar{\bar{u}}_{m-3}(|x^{3}-x^{2}|)=\bar{\bar{u}}_{m-3}(r_{2})=-c_{2}<0.
\end{align}

Once again, $\bar{\bar{\bar{u}}}$ and $\bar{\bar{\bar{u}}}_{i}~(i=1,2,\ldots,m-1)$ satisfy entirely similar equations as $(\bar{\bar{u}},\bar{\bar{u}}_{1},\cdots,\bar{\bar{u}}_{m-1})$~(see \eqref{sec-sph-aver}). Applying a similar argument as previous we did, we can obtain that
\begin{equation}\label{th-iter-est}
\begin{cases}
~\bar{\bar{\bar{u}}}_{m-1}(r)\leq\bar{\bar{\bar{u}}}_{m-1}(0)<0, \\[2mm]
~\bar{\bar{\bar{u}}}_{m-2}(r)\geq\bar{\bar{\bar{u}}}_{m-2}(0)>0, \\[2mm]
~\bar{\bar{\bar{u}}}_{m-3}(r)\leq\bar{\bar{\bar{u}}}_{m-3}(0)<0,
\end{cases}
\end{equation}
for all $r\geq0$. For the reader's convenience to have a better understanding of the iteration process and for the later use of the induction, we give a brief proof for \eqref{th-iter-est} here.
\vskip0.1cm

First of all, by the fact $-\Delta\bar{\bar{\bar{u}}}_{m-1}(r)\geq0$ for $r>0$, we have $\bar{\bar{\bar{u}}}'_{m-1}(r)\leq0$ holds for all $r>0$. This result together with \eqref{siter-est-m-1} imply that
\begin{align*}
\bar{\bar{\bar{u}}}_{m-1}(r)\leq\bar{\bar{\bar{u}}}_{m-1}(0)=\bar{\bar{u}}_{m-1}(|x^{3}-x^{2}|)=
\bar{\bar{u}}_{m-1}(r_{2})\leq\bar{\bar{u}}_{m-1}(0)\leq-c_{0}<0.
\end{align*}

Then via the relation $-\Delta\bar{\bar{\bar{u}}}_{m-2}=\bar{\bar{\bar{u}}}_{m-1}$ and integrating both sides of the above inequality twice, we derive that
\begin{align}\label{thir-even}
\bar{\bar{\bar{u}}}_{m-2}(r)\geq\bar{\bar{\bar{u}}}_{m-2}(0)+\frac{c_{0}}{2n}r^{2},
\end{align}
where $\bar{\bar{\bar{u}}}_{m-2}(0)=\bar{\bar{u}}_{m-2}(|x^{3}-x^{2}|)=\bar{\bar{u}}_{m-2}(r_{2})\geq
\bar{\bar{u}}_{m-2}(0)+\frac{c_{0}}{2n}r_{2}^{2}>0.$

Finally, using the relation $-\Delta\bar{\bar{\bar{u}}}_{m-3}=\bar{\bar{\bar{u}}}_{m-2}$ together with \eqref{thiter-ori-val} and \eqref{thir-even}, one has $\bar{\bar{\bar{u}}}'_{m-3}(r)\leq0$ and then
\begin{align}\label{thir-new-st}
\bar{\bar{\bar{u}}}_{m-3}(r)\leq\bar{\bar{\bar{u}}}_{m-3}(0)=-c_{2}<0.
\end{align}
This completes the proof of \eqref{th-iter-est}. Furthermore, we point out that by the relation $-\Delta\bar{\bar{\bar{u}}}_{m-4}=\bar{\bar{\bar{u}}}_{m-3}$ together with \eqref{thir-new-st}, we can similarly obtain that
$$\bar{\bar{\bar{u}}}_{m-4}(r)\geq\bar{\bar{\bar{u}}}_{m-4}(0)+\frac{c_{2}}{2n}r^{2}, \ \ \forall \ r\geq0.$$
Hence, we can find a large $r_{3}>0$ such that
\begin{align}\label{thir-ind}
\bar{\bar{\bar{u}}}_{m-4}(r_{3})\geq\frac{c_{2}}{4n}r_{3}^{2}:=c_{3}>0,
\end{align}
and then we can start a new re-center and iteration argument for $\bar{\bar{\bar{\bar{u}}}}$ and $\bar{\bar{\bar{\bar{u}}}}_{i}$.
\medskip

In general, after $m$ steps of re-centers process (denotes the centers by $x^{1},x^{2},\cdots,x^{m}$), if we denote the $m$-th spherical average of $u$ and $u_{i}$ by $\tilde{u}$ and $\tilde{u}_{i}$ respectively, then by induction we can deduce that, for $r\geq0$,
\begin{equation}\label{mth-iter-equ}
\begin{cases}
~-\Delta \tilde{u}_{m-1}(r)=\big\{\frac{u^{p}(x)}{(1+|x|^{2})^{\frac{n+2m}{2}-p\frac{n-2m}{2}}}\big\}^{\thicksim}(r)\geq0, \\[2mm]
~(-1)^{i}\tilde{u}_{m-i}(r)\geq (-1)^{i}\tilde{u}_{m-i}(0)>0, \\[2mm]
~(-1)^{m}\tilde{u}(r)\geq (-1)^{m}\tilde{u}(0)>0,
\end{cases}
\end{equation}
where $i=1,2,\ldots,m-1$.

\medskip

\textbf{Step 4}. We claim that $m$ must be even. Otherwise, if $m$ is odd, the last inequality in \eqref{mth-iter-equ} implies immediately that
$$\tilde{u}(r)\leq\tilde{u}(0)<0,$$
which contradicts the hypothesis $u\geq0$.
\medskip

Consequently, by the above ``re-centers and iteration" process~\big(see \eqref{new-iter-poi}, \eqref{regu-sec} and \eqref{thir-ind}\big) and the fact that $m$ is even, we conclude that
\begin{align}\label{mth-iter-est}
\tilde{u}(0)\geq \frac{C}{4n}|x^{m}-x^{m-1}|^{2},
\end{align}
with some constant $C>0$ independent of $x^{m}$. Choosing $x^{m}\in\mathbb{R}^{n}$ such that $|x^{m}|$ sufficiently large, then for all $r\geq0$, we can obtain that
\begin{align}\label{indu-cont}
\tilde{u}(r)\geq\tilde{u}(0)>C_{n}.
\end{align}

On the other hand, from the second inequality in \eqref{eq1}, we know that $u(x)\leq C_{n}$ for all $x\in \mathbb{R}^{n}$, which yields
$$\tilde{u}(r)\leq C_{n}$$
for all $r\geq0$, this contradicts \eqref{indu-cont}. Therefore, $u_{m-1}\geq0$ must hold and we completes the proof for $u_{m-1}\geq 0$.
\vskip 0.2cm

Similarly, we could also prove that $u_{m-i}\geq0$ for $i=2,3,\cdots,m-1$.
Suppose on the contrary, there exists some $2\leq i\leq m-1$ and some $x^{0}\in\mathbb{R}^{n}$ such that
\begin{equation}\label{cont-mid}
\begin{cases}
~u_{m-1}(x)\geq0, \ u_{m-2}(x)\geq0, \cdots, u_{m-i+1}(x)\geq0, \ \forall~x\in\mathbb{R}^{n}, \\[2mm]
~~u_{m-i}(x^{0})<0.
\end{cases}
\end{equation}

Repeating the similar ``re-centers and iteration" process as above, after $m-i+1$ steps of re-centers (denotes the centers by $\bar{x}^{1},\bar{x}^{2},\cdots,\bar{x}^{m-i+1}$), the signs of $\tilde{u}_{m-i+1-j}~(j=1,\cdots,m-i)$ and $\tilde{u}$ satisfy
\begin{equation}\label{reiter-cont-mid}
\begin{cases}
~(-1)^{j}\tilde{u}_{m-i+1-j}(r)\geq(-1)^{j}\tilde{u}_{m-i+1-j}(0)>0, \ \ j=1, \cdots,m-i,\\[2mm]
~~(-1)^{m-i+1}\tilde{u}(r)\geq(-1)^{m-i+1}\tilde{u}(0)>0,
\end{cases}
\end{equation}
for all $r\geq0$. Since $u\geq0$, it follows from \eqref{reiter-cont-mid} that $m-i+1$ is even and
$$\tilde{u}(r)\geq\tilde{u}(0)>0, \ \forall~r\geq0.$$
Moreover, since $m-i$ is odd, then from \eqref{reiter-cont-mid}, we get
$$-\Delta\tilde{u}(r)=\tilde{u}_{1}(r)\leq\tilde{u}_{1}(0)=:-\tilde{c}<0, \ \forall~r\geq0.$$
Integrating both sides of the above inequality twice, we derive that
\begin{align}\label{reiter-est-cont}
\tilde{u}(r)\geq\tilde{u}(0)+\frac{\tilde{c}}{2n}r^{2}\geq\frac{\tilde{c}}{2n}r^{2}, \ \forall~r\geq0.
\end{align}
Therefore, if we assume that \eqref{eq1} holds, we will get a contradiction from \eqref{reiter-est-cont}. As a consequence, $u_{m-i}\geq0$ for all $i=2,3,\cdots,m-1$, this together with $u_{m-1}\geq 0$ ends the proof of the super poly-harmonic properties.
\medskip

Now, we are in position to prove that any nonnegative solution of the high-order elliptic equation \eqref{eq1} in $\mathbb{R}^n$ must be radially decreasing about the origin for $1\leq p<\frac{n+2m}{n-2m}$. From the previous arguments, we know that the nonnegative solution of \eqref{eq1} satisfies the integral equation \eqref{int1}. In what follows, we only need to prove that the nonnegative solution of integral equation \eqref{int1} is radially symmetric and monotone decreasing about the origin.
\vskip0.2cm

\textbf{The proof of the symmetry:}
Our method is based on the moving plane method in integral forms and without loss of generality, we carry out the process of moving plane in the $x_{1}$-direction. Before starting the proof, we recall some definitions.

For $x\in\mathbb{R}^{n}$ and $\lambda\in\mathbb{R}$, we set
$$T_{\lambda}:=\{x=(x_{1},x_{2},\ldots,x_{n})\in\mathbb{R}^{n}: \ x_{1}=\lambda\}$$
and
$$\Sigma_{\lambda}:=\{x=(x_{1},x_{2},\ldots,x_{n})\in\mathbb{R}^{n}: \ x_{1}<\lambda\}.$$

Also, for any $\lambda\in\mathbb{R}$, we let
$$x^{\lambda}:=(2\lambda-x_{1},x_{2},\cdots,x_{n}),$$
be the reflection of the point $x=(x_{1},x_{2},\ldots,x_{n})$ about the plane $T_{\lambda}$.

Define now $u_{\lambda}(x)=u(x^{\lambda})$, then by the fact that $|x-y^{\lambda}|=|x^{\lambda}-y|$, we deduce that
\begin{align}\label{equ-resp}
u(x)=&\Big\{\int_{\Sigma_{\lambda}}+\int_{\mathbb{R}^{n}\backslash\Sigma_{\lambda}}\Big\}\frac{u^{p}(y)}
{|x-y|^{n-2m}(1+|y|^{2})^{\frac{n+2m}{2}-p\frac{n-2m}{2}}}dy \nonumber\\
=&\int_{\Sigma_{\lambda}}\frac{1}{|x-y|^{n-2m}}\frac{u^{p}(y)}{(1+|y|^{2})^{\frac{n+2m}{2}-p\frac{n-2m}{2}}}dy
\nonumber\\
&+\int_{\Sigma_{\lambda}}\frac{1}{|x-y^{\lambda}|^{n-2m}}\frac{u^{p}(y^{\lambda})}{(1+|y^{\lambda}|^{2})
^{\frac{n+2m}{2}-p\frac{n-2m}{2}}}dy \nonumber\\
=&\int_{\Sigma_{\lambda}}\frac{u^{p}(y)}
{|x-y|^{n-2m}(1+|y|^{2})^{\frac{n+2m}{2}-p\frac{n-2m}{2}}}dy \nonumber\\
&+\int_{\Sigma_{\lambda}}\frac{u^{p}(y^{\lambda})}
{|x^{\lambda}-y|^{n-2m}(1+|y^{\lambda}|^{2})^{\frac{n+2m}{2}-p\frac{n-2m}{2}}}dy.
\end{align}
Similarly,
\begin{align}\label{sym-equ-resp}
u(x^{\lambda})=&\int_{\Sigma_{\lambda}}\frac{u^{p}(y)}
{|x^{\lambda}-y|^{n-2m}(1+|y|^{2})^{\frac{n+2m}{2}-p\frac{n-2m}{2}}}dy \nonumber\\
&+\int_{\Sigma_{\lambda}}\frac{u^{p}(y^{\lambda})}
{|x-y|^{n-2m}(1+|y^{\lambda}|^{2})^{\frac{n+2m}{2}-p\frac{n-2m}{2}}}dy.
\end{align}
Combining \eqref{equ-resp} and \eqref{sym-equ-resp}, we obtain that
\begin{align}\label{coup-resp}
u(x)-u(x^{\lambda})=&\int_{\Sigma_{\lambda}}\Big(\frac{1}{|x-y|^{n-2m}}-\frac{1}{|x^{\lambda}-y|^{n-2m}}\Big)
\Big(\frac{u^{p}(y)}{(1+|y|^{2})^{\frac{n+2m}{2}-p\frac{n-2m}{2}}}  \nonumber\\
&-\frac{u^{p}(y^{\lambda})}{(1+|y^{\lambda}|^{2})^{\frac{n+2m}{2}-p\frac{n-2m}{2}}}\Big)dy.
\end{align}

Now, we are in position to prove that the nonnegative solutions of integral equation \eqref{int1} are symmetric with respect to the plane $T_{0}=\{x\in\mathbb{R}^{n}: x_{1}=0\}$.
\begin{proof} The proof is divided into the following two steps.
\medskip

\textbf{Step 1:} We show that for $\lambda$ sufficiently negative, $u_{\lambda}(x)\geq u(x)$ in $\Sigma_{\lambda}$. For this, we let $w_{\lambda}(x)=u_{\lambda}(x)-u(x)$ and define $$\Sigma_{\lambda}^{-}=\{x\in\Sigma_{\lambda}: \ w_{\lambda}(x)<0\},$$
then it suffices for us to prove that for sufficiently negative $\lambda$, $\Sigma_{\lambda}^{-}$ must be measure zero. Indeed, by \eqref{coup-resp}, for $x\in\Sigma_{\lambda}$, we get
\begin{align}\label{est-coup-res}
u(x)-u(x^{\lambda})=&\int_{\Sigma_{\lambda}}\Big(\frac{1}{|x-y|^{n-2m}}-\frac{1}{|x^{\lambda}-y|^{n-2m}}\Big)
\Big(\frac{u^{p}(y)}{(1+|y|^{2})^{\frac{n+2m}{2}-p\frac{n-2m}{2}}}  \nonumber\\
&-\frac{u^{p}(y^{\lambda})}{(1+|y^{\lambda}|^{2})^{\frac{n+2m}{2}-p\frac{n-2m}{2}}}\Big)dy  \nonumber\\
\leq&\int_{\Sigma_{\lambda}}\Big(\frac{1}{|x-y|^{n-2m}}-\frac{1}{|x^{\lambda}-y|^{n-2m}}\Big)\cdot\frac{u^{p}(y)
-u^{p}(y^{\lambda})}{(1+|y|^{2})^{\frac{n+2m}{2}-p\frac{n-2m}{2}}}dy \nonumber\\
=&\int_{\Sigma_{\lambda}}\Big(\frac{1}{|x-y|^{n-2m}}-\frac{1}{|x^{\lambda}-y|^{n-2m}}\Big)\cdot\frac{p\psi_{\lambda}(y)
^{p-1}(u(y)-u(y^{\lambda}))
}{(1+|y|^{2})^{\frac{n+2m}{2}-p\frac{n-2m}{2}}}dy \nonumber\\
\leq&C\int_{\Sigma_{\lambda}^{-}}\frac{1}{|x-y|^{n-2m}}\frac{u^{p-1}(y)(u(y)-u(y^{\lambda}))
}{(1+|y|^{2})^{\frac{n+2m}{2}-p\frac{n-2m}{2}}}dy,
\end{align}
where $\psi_{\lambda}(x)$ lies in between $u_{\lambda}(x)$ and $u(x)$ by the Mean Value Theorem. The last inequality in \eqref{est-coup-res} holds since on $\Sigma_{\lambda}^{-}$, $u_{\lambda}(x)<u(x)$ and then $\psi_{\lambda}(x)\leq u(x)$. Now, notice that for $x\in\Sigma_{\lambda}^{-}$,
\begin{align*}
0<-w_{\lambda}(x)=&u(x)-u_{\lambda}(x).
\end{align*}
Then applying the Hardy-Littlewood-Sobolev inequality to \eqref{est-coup-res}, one gets, for any $\frac{n}{n-2m}<q<\infty$,
\begin{align}\label{norm-est}
\|w_{\lambda}\|_{L^{q}(\Sigma_{\lambda}^{-})}\leq&C\Big\|\int_{\Sigma_{\lambda}^{-}}\frac{u^{p-1}(y)w_{\lambda}(y)
}{|x-y|^{n-2m}(1+|y|^{2})^{\frac{n+2m}{2}-p\frac{n-2m}{2}}}dy\Big\|_{L^{q}(\Sigma_{\lambda}^{-})}  \nonumber\\
\leq&C\Big\|\frac{u^{p-1}(y)w_{\lambda}(y)
}{(1+|y|^{2})^{\frac{n+2m}{2}-p\frac{n-2m}{2}}}\Big\|_{L^{\frac{nq}{n+2mq}}(\Sigma_{\lambda}^{-})} \nonumber\\
\leq&C\Big\|\frac{u^{p-1}(y)}{(1+|y|^{2})^{\frac{n+2m}{2}-p\frac{n-2m}{2}}}
\Big\|_{L^{\frac{n}{2m}}(\Sigma_{\lambda}^{-})}\cdot\|w_{\lambda}\|_{L^{q}(\Sigma_{\lambda}^{-})} \nonumber\\
=&C\Big\{\int_{\Sigma_{\lambda}^{-}}\Big[\frac{u^{p-1}(y)}{(1+|y|^{2})^{\frac{n+2m}{2}-p\frac{n-2m}{2}}}\Big]^
{\frac{n}{2m}}dy\Big\}^{\frac{2m}{n}}\cdot\|w_{\lambda}\|_{L^{q}(\Sigma_{\lambda}^{-})}.
\end{align}
From \eqref{eq1}, we know that
\begin{align}\label{bou-norm-est}
\int_{\mathbb{R}^{n}}\Big[\frac{u^{p-1}(y)}{(1+|y|^{2})^{\frac{n+2m}{2}-p\frac{n-2m}{2}}}\Big]^
{\frac{n}{2m}}dy\leq\int_{\mathbb{R}^{n}}\frac{1}{(1+|y|^{2})^{2m\cdot\frac{n}{2m}}}dy<\infty.
\end{align}
Thus we can choose $N$ sufficiently large, such that for $\lambda\leq-N$,
\begin{align}\label{upp-boun-est}
C\Big\{\int_{\Sigma_{\lambda}^{-}}\Big[\frac{u^{p-1}(y)}{(1+|y|^{2})^{\frac{n+2m}{2}-p\frac{n-2m}{2}}}\Big]^
{\frac{n}{2m}}dy\Big\}^{\frac{2m}{n}}\leq\frac{1}{2}.
\end{align}
Consequently, by \eqref{norm-est} and \eqref{upp-boun-est}, we deduce that $\|w_{\lambda}\|_{L^{q}(\Sigma_{\lambda}^{-})}=0$, which implies that $\Sigma_{\lambda}^{-}$ must be measure zero. This proves that for $\lambda$ sufficiently negative, $w_{\lambda}(x)\geq0$ in $\Sigma_{\lambda}$.
\medskip

\textbf{Step 2:} Move the plane to the limiting position to derive symmetry and monotonicity.
Now we start to move the plane $T_{\lambda}=\{x\in\mathbb{R}^{n}: x_{1}=\lambda\}$ to the right as long as $w_{\lambda}(x)\geq 0$ holds in $\Sigma_{\lambda}$ and define
\begin{align}\label{sup-mov-pla}
\lambda_{0}=\sup\{\lambda: w_{\mu}(x)\geq0, \ \forall~x\in\Sigma_{\mu}, \ \mu\leq\lambda\}.
\end{align}
From the \textbf{Step 1}, we know that $\lambda_0$ is well defined. In addition, according to the definition of $\lambda_0$, one can easily get $w_{\lambda_0}(x)\geq 0$ in $\Sigma_{\lambda_0}$.
Now, we claim that $\lambda_0\geq 0$. To prove it, we argue by contradiction and suppose that $\lambda_{0}<0$. In this case, we first show that
\begin{align}\label{hyp-ind}
w_{\lambda_{0}}(x)\equiv0 \  \  in \ \Sigma_{\lambda_{0}}.
\end{align}
If not, then by the definition of $\lambda_{0}$, we can assume that $w_{\lambda_{0}}(x)\geq0$, but $w_{\lambda_{0}}(x)\not\equiv0$ in $\Sigma_{\lambda_{0}}$. In particular, we remark that under this assumption, we have in fact $w_{\lambda_{0}}(x)>0$ in the interior of $\Sigma_{\lambda_{0}}$.
Indeed, if there exists $x_{1}\in\Sigma_{\lambda_{0}}$ such that $w_{\lambda_{0}}(x_{1})=0$, then by \eqref{coup-resp}, we have
\begin{align*}
0=w_{\lambda_{0}}(x_{1})=&\int_{\Sigma_{\lambda_{0}}}\Big(\frac{1}{|x_{1}-y|^{n-2m}}-\frac{1}{|x_{1}^{\lambda_{0}}
-y|^{n-2m}}\Big)\Big(\frac{u^{p}(y^{\lambda_{0}})}{(1+|y^{\lambda_{0}}|^{2})^{\frac{n+2m}{2}-p\frac{n-2m}{2}}}  \\
&-\frac{u^{p}(y)}{(1+|y|^{2})^{\frac{n+2m}{2}-p\frac{n-2m}{2}}}\Big)dy \\
\geq&\int_{\Sigma_{\lambda_{0}}}\Big(\frac{1}{|x_{1}-y|^{n-2m}}-\frac{1}{|x_{1}^{\lambda_{0}}
-y|^{n-2m}}\Big)\frac{u^{p}(y^{\lambda_{0}})-u^{p}(y)}{(1+|y|^{2})^{\frac{n+2m}{2}-p\frac{n-2m}{2}}}dy.
\end{align*}

On the other hand, by the hypothesis on $w_{\lambda_{0}}$, we know that there exists at least one point $x_{2}\in\Sigma_{\lambda_{0}}$ such that
$$w_{\lambda_{0}}(x_{2})=u_{\lambda_{0}}(x_{2})-u(x_{2})>0.$$
Then by the uniform continuity of $u$ and $u_{\lambda_{0}}$, we can find a small neighborhood $B_{r}(x_{2})\subset\Sigma_{\lambda_{0}}$ such that $w_{\lambda_{0}}(x)>0$ holds for all $x\in B_{r}(x_{2})$. As a consequence, we deduce that
\begin{align*}
0=w_{\lambda_{0}}(x_{1})\geq&\int_{B_{r}(x_{2})}\Big(\frac{1}{|x_{1}-y|^{n-2m}}  \\
&-\frac{1}{|x_{1}^{\lambda_{0}}
-y|^{n-2m}}\Big)\frac{u^{p}(y^{\lambda_{0}})-u^{p}(y)}{(1+|y|^{2})^{\frac{n+2m}{2}-p\frac{n-2m}{2}}}dy>0,
\end{align*}
this is impossible and the conclusion $w_{\lambda_{0}}(x)>0$ in $\Sigma_{\lambda_{0}}$ follows.
\vskip0.2cm

In what follows we show that the plane $T_{\lambda}$ can be moved a little further to the right provided that $w_{\lambda_{0}}(x)>0$ in $\Sigma_{\lambda_{0}}$, more precisely, we prove that there exists some $\bar{\varepsilon}>0$ small enough such that for any $0\leq\varepsilon\leq\bar{\varepsilon}$, $w_{\lambda}(x)\geq0$ holds in $\Sigma_{\lambda}$ for all $\lambda\in[\lambda_{0},\lambda_{0}+\varepsilon)$, which will contradict with the definition of $\lambda_0$.
In order to prove it, we shall follow the arguments used in \textbf{Step 1}, that is, we check that \eqref{upp-boun-est} holds and then obtain $\|w_{\lambda}\|_{L^{q}(\Sigma_{\lambda}^{-})}=0$ for all $\lambda\in [\lambda_{0},\lambda_{0}+\varepsilon)$.
\vskip 0.2cm

We now prove that the inequality in \eqref{upp-boun-est} holds for $\lambda\in[\lambda_{0},\lambda_{0}+\varepsilon)$. First of all, from \eqref{bou-norm-est}, we can choose $R$ sufficiently large such that
\begin{align}\label{unbou-tru-est}
\int_{\mathbb{R}^{n}\backslash B_{R}(0)}\Big[\frac{u^{p-1}(y)}{(1+|y|^{2})^{\frac{n+2m}{2}-p\frac{n-2m}{2}}}\Big]^
{\frac{n}{2m}}dy<\frac{1}{4C}.
\end{align}
Fix this $R$, then we only need to show that the measure of $\Sigma_{\lambda}^{-}\cap B_{R}(0)$ is sufficiently small as $\lambda$ close to $\lambda_{0}$.

To verify it, for $\lambda>\lambda_{0}$ and $\delta>0$, we define
$$E_{\delta}=\{x\in\Sigma_{\lambda_{0}}\cap B_{R}(0): \ w_{\lambda_{0}}(x)>\delta\},$$
$$F_{\delta}=(\Sigma_{\lambda_{0}}\cap B_{R}(0))\backslash E_{\delta}, \ \ D_{\lambda}=(\Sigma_{\lambda}\backslash\Sigma_{\lambda_{0}})\cap B_{R}(0).$$
Then it is easy to see that
\begin{align}\label{con-mea}
(\Sigma_{\lambda}^{-}\cap B_{R}(0))\subset (\Sigma_{\lambda}^{-}\cap E_{\delta})\cup F_{\delta}\cup D_{\lambda},
\end{align}
and
\begin{align}\label{inf-str-mea}
\lim_{\lambda\rightarrow\lambda_{0}^{+}}\mu(D_{\lambda})=0.
\end{align}
Moreover, by the fact that $w_{\lambda_{0}}(x)>0$ in $\Sigma_{\lambda_{0}}$, we can obtain that
\begin{align}\label{inf-mea-val}
\lim_{\delta\rightarrow0}\mu(F_{\delta})=0.
\end{align}

Now, for any fixed $\eta>0$, we can choose a $\delta>0$ small enough such that $\mu(F_{\delta})\leq\eta$. Fix this $\delta$, next we prove that the measure of $\Sigma_{\lambda}^{-}\cap E_{\delta}$ can also be sufficiently small as $\lambda$ close to $\lambda_{0}$. In fact, for arbitrary $x\in\Sigma_{\lambda}^{-}\cap E_{\delta}$, we have
$$0>w_{\lambda}(x)=u(x^{\lambda})-u(x^{\lambda_{0}})+u(x^{\lambda_{0}})-u(x),$$
which implies that
$$u(x^{\lambda_{0}})-u(x^{\lambda})>u(x^{\lambda_{0}})-u(x)>\delta.$$
Then it follows that
\begin{align}\label{con-set}
(\Sigma_{\lambda}^{-}\cap E_{\delta})\subset G_{\delta}:=\{x\in B_{R}(0): \ u(x^{\lambda_{0}})-u(x^{\lambda})>\delta\}.
\end{align}
By the well known Chebyshev inequality, we have
\begin{align}\label{Che}
\mu(G_{\delta})\leq&\frac{1}{\delta^{r}}\int_{G_{\delta}}|u(x^{\lambda_{0}})-u(x^{\lambda})|^{r}dx \nonumber\\
\leq&\frac{1}{\delta^{r}}\int_{B_{R}(0)}|u(x^{\lambda_{0}})-u(x^{\lambda})|^{r}dx
\end{align}
for any $1\leq r<\frac{n}{n-2m}$. Hence, for the fixed $\delta$, as $\lambda\rightarrow\lambda_{0}^{+}$, the right hand of the above inequality can be made as small as we wish.

Therefore, by \eqref{con-mea}, \eqref{inf-str-mea}, \eqref{inf-mea-val}, \eqref{con-set} and \eqref{Che}, we obtain
$$\lim_{\lambda\rightarrow\lambda_{0}^{+}}\mu(\Sigma_{\lambda}^{-}\cap B_{R}(0))\leq\mu(F_{\delta})\leq\eta.$$
Since $\eta>0$ is arbitrarily chosen, this gives
\begin{align}\label{zer-lim-mea}
\lim_{\lambda\rightarrow\lambda^{+}_{0}}\mu(\Sigma_{\lambda}^{-}\cap B_{R}(0))=0.
\end{align}

Combining \eqref{unbou-tru-est} and \eqref{zer-lim-mea}, we arrive at \eqref{upp-boun-est} for $\lambda\in[\lambda_{0},\lambda_{0}+\varepsilon)$ with $\varepsilon>0$ sufficiently small.
As a consequence, there exists some $\bar{\varepsilon}>0$ small enough such that for all $0\leq\varepsilon\leq\bar{\varepsilon}$ and $\lambda\in[\lambda_{0},\lambda_{0}+\varepsilon)$, $\|w_{\lambda}\|_{L^{q}(\Sigma_{\lambda}^{-})}=0$. Thus $\mu(\Sigma_{\lambda}^{-})=0$, i.e.,
$$w_{\lambda}(x)\geq0, \ a.e. ~\forall \ x\in\Sigma_{\lambda}.$$
This is a contradiction with \eqref{sup-mov-pla}, therefore \eqref{hyp-ind} must hold. Now, by \eqref{coup-resp}, \eqref{hyp-ind} and the assumption that $\lambda_{0}<0$, we can write for any $x\in\Sigma_{\lambda_{0}}$,
\begin{align*}
0=w_{\lambda_{0}}(x)=&u(x)-u(x^{\lambda_{0}}) \nonumber\\
=&\int_{\Sigma_{\lambda_{0}}}\Big(\frac{1}{|x-y|^{n-2m}}-\frac{1}{|x^{\lambda_{0}}-y|^{n-2m}}\Big)
\Big(\frac{u^{p}(y)}{(1+|y|^{2})^{\frac{n+2m}{2}-p\frac{n-2m}{2}}}  \nonumber\\
&-\frac{u^{p}(y^{\lambda_{0}})}{(1+|y^{\lambda_{0}}|^{2})^{\frac{n+2m}{2}-p\frac{n-2m}{2}}}\Big)dy
<0,
\end{align*}
which is a contradiction. Thus we have $\lambda_{0}\geq 0$ and then
$$u(2\lambda_{0}-x_{1},x_{2},\cdots,x_{n})\geq u(x_{1},x_{2},\cdots,x_{n}), \ \forall~x\in\Sigma_{\lambda_0}.$$

On the other hand, by using a similar argument as above, we can also move the plane from $x_{1}=+\infty$ to the left, which yields that $\lambda_{0}\leq 0$ and then
$$u(2\lambda_{0}-x_{1},x_{2},\cdots,x_{n})\leq u(x_{1},x_{2},\cdots,x_{n}), \ \forall~x\in\Sigma_{\lambda_0}.$$
Consequently, $\lambda_{0}=0$ and we have that
$$u(-x_{1},x_{2},\cdots,x_{n})\equiv u(x_{1},x_{2},\cdots,x_{n}), \ \forall~x\in\mathbb{R}^{n},$$
that is, $u(x)$ is symmetric with respect to the plane $T_{0}$. Since the equation is invariant under rotation, the direction of $x_{1}$ can be chosen arbitrarily. Then we deduce that the nonnegative solution $u(x)$ must be radially symmetric and monotone decreasing about the origin $0\in\mathbb{R}^{n}$. This completes the proof of the symmetry.

\end{proof}

\section{The proof of Theorem \ref{thm2}}\label{section3}
In this section, we will give the classification of the extremals for the subcritical high-order Sobolev inequality \eqref{sub} in $\mathbb{S}^n$. Obviously, up to a constant, the extremals of high-order Sobolev inequality \eqref{sub} satisfy the following Euler-Lagrange equation
\begin{equation*}
P_{m,g_{\mathbb{S}^n}}(u)=u^{p}\  {\rm in}\ \mathbb{S}^n,
\end{equation*}
that is, equation \eqref{hlane}. Hence, if we can prove that the extremals of high-order Sobolev inequality \eqref{sub} are non-negative, then applying the Liouville-type result of Theorem \ref{thm1}, we can classify the extremals of the high-order Sobolev inequality \eqref{sub} and compute the corresponding sharp constant. Next, in order to finish the proof of Theorem \ref{thm2}, it suffices for us to prove the nonnegativity of the extremals.
\medskip

\begin{proof}
From the subcritical high-order Sobolev inequality \eqref{sub}, we know that for any $u\in W^{m,2}(\mathbb{S}^n)$,
\begin{equation}\label{in1}
\int_{\mathbb{S}^n}P_{m,g_{\mathbb{S}^n}}(u)ud\sigma\geq S_{m,n,p} \big(\int_{\mathbb{S}^n}|u|^pd\sigma\big)^{\frac{2}{p}}.
\end{equation}
Let $v=P_{m,g_{\mathbb{S}^n}}^{\frac{1}{2}}(u)$, then inequality \eqref{in1} is equivalent to the following inequality
\begin{equation}\label{in2}
\int_{\mathbb{S}^n}v^2d\sigma\geq S_{m,n,p} \big(\int_{\mathbb{S}^n}|P_{m,g_{\mathbb{S}^n}}^{-\frac{1}{2}}(v)|^pd\sigma\big)^{\frac{2}{p}}.
\end{equation}

On the other hand, by the dual argument, one can deduce that the inequality \eqref{in2} is equivalent to
\begin{equation}\label{in3}
\big(\int_{\mathbb{S}^n}v^{p'}d\sigma\big)^{\frac{2}{p'}}\geq S_{m,n,p} \big(\int_{\mathbb{S}^n}P_{m,g_{\mathbb{S}^n}}^{-1}(v)vd\sigma\big).
\end{equation}
Since $P_{m,g_{\mathbb{S}^n}}^{-1}(v)(\xi)=\int_{S^n}\frac{v(\eta)}{|\xi-\eta|^{n-2m}}d\sigma_{\eta}$, the above inequality is equivalent to the following Hardy-Littlewood-Sobolev inequality on $\mathbb{S}^n$,
\begin{equation}\label{in3}
S_{m,n,p}\int_{\mathbb{S}^n}\int_{\mathbb{S}^n}\frac{v(\xi)v(\eta)}{|\xi-\eta|^{n-2m}}d\sigma_{\xi}d\sigma_{\eta}\leq \big(\int_{\mathbb{S}^n}v^{p'}d\sigma\big)^{\frac{2}{p'}}.
\end{equation}
As is well known, the extremals of Hardy-Littlewood-Sobolev inequality on $\mathbb{S}^n$ must be non-negative, hence we conclude that the extremals of the high-order Sobolev inequalities \eqref{sub} must be non-negative. This together with the Liouville type result we obtain in Theorem \ref{thm1} accomplishes the proof of Theorem \ref{thm2}.
\end{proof}


\end{document}